\documentclass[graybox]{svmult}
\usepackage{type1cm}        %

\usepackage{makeidx}         
\usepackage{graphicx}        
\usepackage{multicol}        
\usepackage[bottom]{footmisc}
\usepackage{newtxtext}       %
\usepackage[varvw]{newtxmath}

\usepackage{array,tikz,cite,hyperref,microtype,multirow}
\newcolumntype{M}[1]{>{\centering\arraybackslash}m{#1}}
\newcolumntype{L}[1]{>{\centering\arraybackslash}m{#1}} 
\newcolumntype{P}[1]{>{\centering\arraybackslash}p{#1}} 
\makeindex 
\allowdisplaybreaks

\usepackage{tabularx,booktabs,orcidlink}

\usepackage{mathtools,enumitem,tabularray,array,subcaption,bm}
\usepackage[capitalise,nameinlink,noabbrev]{cleveref}

\usepackage{subcaption}

\AtBeginEnvironment{align}{
  \setlength{\abovedisplayskip}{4pt}    
  \setlength{\belowdisplayskip}{4pt}    
  \setlength{\abovedisplayshortskip}{0pt}
  \setlength{\belowdisplayshortskip}{0pt}
  \setlength{\jot}{1pt}                 
}

\definecolor{ricam}{RGB}{0, 68, 179}
\definecolor{math}{RGB}{107, 153, 0}

\def\div{\operatorname{div}}

\def\Th{\mathcal{T}_h}
\def\Itau{\mathcal{I}_\tau}
\def\I{\mathbf{I}}

\def\dx{\,\textrm{d}x}
\def\la{\langle}
\def\ra{\rangle}

\def\Th{\mathcal{T}_h}
\def\Vh{\mathcal{V}_h}
\def\Qh{\mathcal{Q}_h}
\def\Xh{\mathcal{X}_h}
\def\Zh{\mathcal{Z}_h}
\def\Wh{\mathcal{W}_h}

\def\vv{{\bm{v}}}

\DeclarePairedDelimiter{\snorm}{|}{|}
\DeclarePairedDelimiter{\norm}{\|}{\|}

\def\E{\mathcal{E}}

\def\w{\mathbf{w}}

\def\dt{\partial_t}

\def\dtau{d_\tau^{n+1}}
\def\avg{\textup{avg}}

\def\ddt{\frac{\mathrm{d}}{\mathrm{d}t}}

\def\softd{{\leavevmode\setbox1=\hbox{d}%
		\hbox to 1.05\wd1{d\kern-0.4ex{\char039}\hss}}}

\begin{document}
\title*{Review of thermodynamic structures and structure-preserving discretisations of Cahn--Hilliard-type models}
\titlerunning{Review of structure-preserving discretisations of Cahn--Hilliard-type models}
\author{Aaron Brunk\orcidlink{0000-0003-4987-2398},~ 
Marco F.P. ten Eikelder\orcidlink{0000-0002-1153-146X},~
Marvin Fritz\orcidlink{0000-0002-8360-7371},~
Dennis Höhn\orcidlink{0009-0006-4853-9947
},~ and
Dennis Trautwein\orcidlink{0009-0000-5136-4566}}

\institute{Aaron Brunk \at Johannes Gutenberg University Mainz, \email{abrunk@uni-mainz.de}
\and Marco F.P. ten Eikelder \at Technical University of Darmstadt, \email{marco.eikelder@tu-darmstadt.de}
\and Marvin Fritz \at Radon Institute for Computational and Applied Mathematics, \email{marvin.fritz@oeaw.ac.at}
\and Dennis Höhn \at Johannes Gutenberg University Mainz, \email{dennis.hoehn@uni-mainz.de}
\and Dennis Trautwein \at University of Regensburg, \email{dennis.trautwein@ur.de}
}

\authorrunning{A.\,Brunk; M.F.P.\,ten Eikelder; M.\,Fritz; D.\,Höhn; D.\,Trautwein}

\maketitle

\abstract{The Cahn–Hilliard equation and extensions, notably the Cahn–Hilliard–Darcy and Cahn–Hilliard–Navier–Stokes systems, provide widely used frameworks for coupling interfacial thermodynamics with flow. This review surveys the thermodynamic structures underlying these models, focusing on the formulation of free energy functionals, dissipation mechanisms, and variational principles. We compare structural properties, emphasizing how these models encode conservation laws and energy dissipation. A central theme is the translation of these thermodynamic structures into numerical practice by providing representative discretisation strategies that aim to preserve mass conservation, stability, and energy decay. Particular attention is paid to the trade-offs between accuracy, efficiency, and structure preservation in large-scale simulations.}

\section{Introduction}
Phase-field models have become a cornerstone in the mathematical description of multiphase systems, with the Cahn--Hilliard (CH) equation and its extensions serving as prototypical examples; we refer to the review articles \cite{novick2008cahn,kim2016basic,wu2022review,ten2023unified,fritz2023tumor,ten2025unified} and the references therein. These models couple thermodynamic consistency with interfacial dynamics and, when combined with fluid flow, give rise to the Cahn--Hilliard--Darcy (CHD) and Cahn--Hilliard--Navier--Stokes (CHNS) systems. A unifying feature of such models is the presence of underlying thermodynamic structures, notably conservation laws and energy dissipation, which are essential for both physical interpretability and long-time stability.

Transferring these structures to the discrete setting has been the subject of extensive research, leading to the development of structure-preserving numerical schemes. Such methods aim to retain key properties such as mass conservation and energy decay, thereby providing more reliable simulations than conventional discretisations. This survey highlights the thermodynamic framework underlying CH-type models, reviews representative discretisation strategies, and discusses the trade-offs between accuracy, efficiency, and structure preservation. The emphasis is on presenting common structural principles and illustrating how they inform the design of robust numerical methods. 

Beyond the CH, CHD, and CHNS systems considered here, structure-preserving methods have been studied for Cahn--Hilliard type models coupled with Biot’s equation \cite{brunk2025structure}, nonlocal interactions \cite{brunk2024analysis}, magnetohydrodynamics \cite{shi2024structure}, Poisson--Nernst--Planck dynamics \cite{qian2023convergence}, reaction--diffusion systems \cite{wang2025stability}, Forchheimer's equation \cite{brunk2025analysis} and lymphangiogenesis models \cite{wang2025structure,jungel2024structure}. 

\section{Models}
The mixed forms of the Cahn--Hilliard type models read:
\begin{center}\textbf{Cahn--Hilliard (CH) system} \end{center}
\begin{align}
  \dt\phi - \div(m(\phi)\nabla\mu) &= 0, \qquad \mu + \gamma\Delta\phi - f'(\phi) = 0\label{eq:ch}
\end{align}

\begin{center}\textbf{Cahn--Hilliard--Darcy (CHD) system} \end{center}
\begin{subequations}
\begin{align}
  \dt\phi + \div(\phi\vv) - \div(m(\phi)\nabla\mu) &= 0,  \qquad \mu + \gamma\Delta\phi - f'(\phi)  = 0,\label{eq:chd1}\\
  \alpha(\phi)\vv + \nabla p + \div(\gamma\nabla\phi\otimes\nabla\phi) &= 0, \qquad \div(\vv) = 0, \label{eq:chd2}
  \end{align}
\end{subequations}

\begin{center}\textbf{Cahn--Hilliard--Navier--Stokes (CHNS) system} \end{center}
\begin{subequations}
\begin{align}
  \dt\phi &+ \div(\phi\vv) - \div(m(\phi)\nabla\mu) = 0, \qquad \mu + \gamma\Delta\phi - f'(\phi)  = 0, \label{eq:chns1}\\
  \dt\vv &+ \div(\vv\otimes\vv) - \div(\eta(\phi)\mathrm{D}\vv) + \nabla p + \div(\gamma\nabla\phi\otimes\nabla\phi) = 0, \qquad \div(\vv) = 0,\label{eq:chns2}
\end{align}
\end{subequations}
in domain $\Omega \subset \mathbb{R}^d$ (dimension $d=2,3$). In these models, $\phi \in [0,1]$ denotes the phase field variable, which represents an order parameter (e.g. volume fraction, mass fraction or concentration), $\mu$ is the chemical potential, $\vv$ is the mixture velocity, and $p$ is the Lagrange multiplier pressure. Furthermore, $m=m(\phi)$ is the mobility function, $f=f(\phi)$ is the gradient-free part of the Helmholtz free energy density, $\gamma$ is a parameter related to the interfacial width, and $\eta = \eta(\phi)$ is the viscosity. The CH equation is a diffuse interface model for phase separation in binary mixtures. It evolves an order parameter that captures spinodal decomposition and coarsening (Eq.~\eqref{eq:ch}). The CHD model describes a system that involves both phase separation/coarsening (Eq.~\eqref{eq:chd1}) and porous medium flow (Eq.~\eqref{eq:chd2}). The CHNS system models phase separation (Eq.~\eqref{eq:chns1}) and incompressible hydrodynamics (Eq.~\eqref{eq:chns2}). The systems have to be complemented by boundary and initial conditions. In this review, we consider for simplicity periodic boundary conditions. From a physical perspective, the CHNS system is the most complete model of these three. In certain regimes where inertia is not dominant, certain model reductions/approximations, like the CHD system,
appear to be valid. Whenever the momentum is negligible, the standard Cahn-Hilliard system is a suitable reduction.

We proceed with the physical structures encoded in these systems. We define the mass, linear momentum, angular momentum, and energy, respectively, as: 
\begin{subequations}
    \begin{align*}
      M =~ \int_\Omega \phi \dx,\qquad
      \mathbf{L} =~ \int_\Omega \vv\dx,\qquad
      \mathbf{P} =~ \int_\Omega \vv\times \mathbf{x} \dx,\qquad
      E_\beta =~ \int_\Omega \Psi + \frac{\beta}{2}\snorm{\vv}^2 \dx,
    \end{align*}
\end{subequations}
where the free energy density is $\Psi = \frac{\gamma}{2}\snorm{\nabla\phi}^2 + f(\phi)$, and where $\beta=0$ for CH and CHD, and $\beta=1$ for CHNS. In Table \ref{tab:structures}, we summarise the conservation and dissipation structures of the models. 
\newcolumntype{A}{>{\hsize=3\hsize\raggedright\arraybackslash}X}
\newcolumntype{S}{>{\hsize=14\hsize\raggedright\arraybackslash}X}
\newcolumntype{F}{>{\hsize=27\hsize\raggedright\arraybackslash}X}

\begin{table}[!htbp]
  \centering
  \begin{tabularx}{\textwidth}{@{} A S F @{}}
    \toprule
    Model & Structure & Formula \\
    \midrule
    \multirow{2}{*}{CH}   & Mass conservation  & $\ddt M = 0$ \\[2pt]
                          & Energy dissipation & $\ddt E_0 = - \la m(\phi)\nabla\mu,\nabla\mu \ra \le 0$ \\
    \addlinespace
    \multirow{2}{*}{CHD}  & Mass conservation  & $\ddt M = 0$ \\
                          & Energy dissipation & $\ddt E_0 = - \la \alpha(\phi)\vv,\vv \ra - \la m(\phi)\nabla\mu,\nabla\mu \ra \le 0$ \\
    \addlinespace
    \multirow{3}{*}{CHNS} & Mass conservation  & $\ddt M = 0$ \\
                          & Linear/angular momentum cons. & $\ddt \mathbf{L} = 0$ \qquad  $\ddt \mathbf{P} = 0$ \\
                          & Energy dissipation & $\ddt E_1 = - \la m(\phi)\nabla\mu,\nabla\mu \ra - \la \eta(\phi)\mathrm{D}\vv,\mathrm{D}\vv \ra \le 0$ \\
    \bottomrule
  \end{tabularx}
  \caption{Physical structures for CH, CHD, and CHNS.}
  \label{tab:structures}
\end{table}
The mass conservation equations follow from the integration of the phase-field equation over $\Omega$, subsequently integrating by parts, and applying the periodic boundary conditions. Next, linear and angular momentum conservation in the CHNS model follows analogously as in the classical Navier-Stokes equations, see e.g. a standard textbook on continuum mechanics such as \cite{tadmor2012continuum}. Finally, the energy-dissipation laws result from choosing the appropriate weights, taking the superposition, and integrating. The weights are $\mu$ for the $\phi$ equation, $-\dt\phi$ for the $\mu$, $\vv$ for the $\vv$ equation and $p$ for the divergence equation. We note that to establish the energy-dissipation law of CHD/CHNS, we used the following identity for the Korteweg stress $\div(\gamma\nabla\phi\otimes\nabla\phi) =  \nabla \Psi - \mu \nabla\phi. $ 

\section{Discretisation}
We assume that $\Omega \subset \mathbb R^d$ is a convex polyhedral domain and $\Th$ is assumed to be a globally quasi-uniform triangulation of $\Omega$. The phase-field and its chemical potential are discretised by continuous piece-wise linear elements, while for velocity-pressure we use standard LBB stable spaces; e.g. Taylor--Hood elements, see \cite{boffi2013mixed} for further details. We denote by $\mathbb{P}_k=\{ v \in L^2(\Omega) : v|_K \in \mathbb{P}_k(K) \ \forall K \in \Th \}$ the finite element space of piecewise polynomials of degree $k$, and by $\mathbb{P}_k^c=\{ v \in H^1(\Omega) : v|_K \in \mathbb{P}_k(K)
     \ \forall K \in \Th \}$ the globally continuous version. We choose the spaces
\begin{align*}
	\Vh &:= \mathbb{P}_1^c, \qquad \Xh :=(\mathbb{P}_2^c)^d, \qquad \Qh := \mathbb{P}_1^c, \qquad \Wh := \mathbb{P}_1\\
    \Zh &:= \left\{ \mathbf{v} \in [L^2(\Omega)]^d : \nabla\cdot \mathbf{v} \in L^2(\Omega),~
\mathbf{v}|_K \in [\mathbb{P}_2(K)]^d \ \ \forall K \in \mathcal{T}_h \right\}.
\end{align*}
We divide the time interval $[0,T]$, $T>0$ fixed, into uniform steps with step size $\tau>0$ and introduce $\Itau:=\{0=t^0,t^1=\tau,\ldots, t^{n_T}=T\}$, where $n_T=\tfrac{T}{\tau}$ is the absolute number of time steps. We denote by $\Pi^1_c(\Itau;X)$ and $\Pi^0(\Itau;X)$ the spaces of continuous piecewise linear and piecewise constant functions on $\Itau$ with values in the space $X$, respectively. By $g^{n+1}$ and $g^n$ we denote the evaluation of a function $g$ in $\Pi^1_c(\Itau;X)$ or $\Pi^0(\Itau;X)$ at $t=\{t^{n+1},t^n\}$, respectively, and write $I_n=(t^n,t^{n+1})$. The discrete-time derivative is denoted by $d^{n+1}_\tau g := \frac{g^{n+1}-g^n}{\tau}$. We note that we use the time-averaged method to approximate the potential derivative, that is,
\begin{align*}
    f'(\phi^{n+1},\phi^n) := \frac{1}{\tau}\int_{t^n}^{t^{n+1}} f'(\phi^{n,n+1}_\avg(s)) \mathrm{d}s, \quad \phi^{n,n+1}_\avg(s) = \frac{\phi^{n+1}-\phi^n}{\tau}(s-t^n) + \phi^n.
\end{align*}

We briefly remark on the above choice of time integration of the potential derivative. In general the double-well potential considered for Cahn-Hilliard systems are neither convex nor concave. Hence, low-order methods like implicit or explicit Euler fail to be unconditionally energy-stable, \cite{Shen2010}. Conditional stability can be achieved using structural assumption on the potential. However, such restrictions are unfavoured latest when considering long-time simulations. In the lowest order case unconditional stability can be achieved using convex-concave decomposition or stabilised methods and similar extensions for higher-order methods \cite{Shen2010,gomez2011provably,GuillnGonzlez2014}. Both methods introduce numerical diffusion in time. Our approach utilise exact integration in time and hence no numerical diffusion is introduced. The method is quite similar to the averaged vector-field method \cite{Gonzalez1996,McLachlan}. 
\subsection{Cahn--Hilliard system}

Let $\phi_{h}^0\in \Vh$ be given; seek functions $(\phi_h,\mu_h)\in \Pi_c^1(\Itau;\Vh)\times\Pi^0(\Itau;\Vh)$ such that
\begin{align}
  \la \dtau\phi_h,\psi_h \ra & + \la m(\phi_h^{n+1})\nabla\mu_h^{n+1},\nabla\psi_h \ra = 0, \label{eq:CHscheme1}
  \\
  \la \mu_h^{n+1},\xi_h \ra &- \gamma\la \nabla\phi_h^{n+1},\nabla\xi_h \ra - \la f'(\phi_h^{n+1},\phi_h^n),\xi_h \ra = 0,\label{eq:CHscheme2}
\end{align}
holds for all $(\psi_h,\xi_h)\in\Vh\times\Vh$ and for all $0\leq n < n_T.$ This scheme is fully implicit in the interfacial terms, uses a time-averaged treatment of the bulk potential, and yields a discrete energy law mirroring the continuous one (see Theorem~\ref{thm:main}).

We note that a broad variety of structure-preserving discretisations have been developed for CH to guarantee (at least) energy stability. Discontinuous Galerkin schemes have been proposed, see \cite{wimmer2025structure,acosta2023upwind}, as well as upwind-based SAV (Scalar Auxiliary Variable) approaches \cite{huang2023structure}. Energy-stable formulations with dynamic boundary conditions are studied in \cite{fukao2017structure,okumura2022second,okumura2024structure}, while further SAV-type schemes have been analysed in, e.g., \cite{metzger2021efficient, metzger2023convergent, metzger2025convergent}. The CH equation has also been reformulated in a port-Hamiltonian framework \cite{bendimerad2023structure, altmann2025structure}, and more recently, machine-learning-based discretisations with built-in structure preservation have been explored \cite{chu2024structure}. Anisotropic variants have been addressed in \cite{li2025structure}, while optimization-based reformulations lead to alternative stable discretisations \cite{ding2023structure}. Finite volume schemes with unconditional energy stability are developed and analysed in \cite{bailo2023unconditional,bailo2020convergence}.

Beyond energy-stable schemes, there are other structures of interest that one could try to preserve. We mention the bound preserving property such as $\phi_h \in [0,1]$ when considering degenerate mobilities and singular potentials, see \cite{bailo2023unconditional,barrett1999finite,ten2024ostwald}. Furthermore, one may prove the attainment of a steady-state, which is true for the continuous model (see \cite{rybka1999convergence,abels2007convergence}), and it was shown in \cite{antonietti2016convergence,merlet2009convergence,brachet2022convergence} for several discretisation schemes.

\subsection{Cahn--Hilliard--Darcy system}

Let $\phi_{h}^0\in \Vh$ be given; seek function $\phi_h\in \Pi_c^1(\Itau;\Vh)$ and $(\mu_h,\vv_h,p_h)\in \Pi^0(\Itau;\Vh\times\Zh\times\Wh)$ such that
\begin{align}
  \la \dtau\phi_h,\psi_h \ra &- \la \phi_h^{n+1}\vv_h^{n+1}, \nabla\psi_h\ra + \la m(\phi_h^{n+1})\nabla\mu_h^{n+1},\nabla\psi_h \ra = 0, \label{eq:CHDscheme1}
  \\
  \la \mu_h^{n+1},\xi_h \ra &- \gamma\la \nabla\phi_h^{n+1},\nabla\xi_h \ra - \la f'(\phi_h^{n+1},\phi_h^n),\xi_h \ra = 0,\label{eq:CHDscheme2}
  \\
  \la \alpha(\phi_h^{n+1})\vv_h^{n+1},\w_h \ra &- \la p_h^{n+1},\div(\w_h)\ra + \la \phi_h^{n+1}\nabla\mu_h^{n+1},\w_h \ra = 0, \label{eq:CHDscheme3}
  \\
  \la \div(\vv_h^{n+1}),q_h \ra & = 0 \label{eq:CHDscheme4}
\end{align}
holds for all $(\psi_h,\xi_h,\w_h,q_h)\in\Vh\times\Vh\times\Zh\times\Wh$ and for all $0\leq n < n_T.$

Similar to the classical CH equation, a broad variety of structure-preserving discretisations have been developed for the CHD system and related variants. Several finite element schemes achieve first- or second-order temporal accuracy, while maintaining energy stability and often allowing decoupled solution strategies \cite{han2016_CHD_decoupled, han_wang2018_CHD, zhang2021_CHD}.
An energy-stable and well-posed FE scheme was also proposed for the related CH--Forchheimer system \cite{brunk2025analysis}. Other FE-based studies provide energy-stable approximations and error estimates for CHD models with additional Stokes-type stabilisation terms \cite{diegel_feng_wise2015_CHD, chen2022_CHD}.
Alternative SAV formulations yield linear, unconditionally energy-stable schemes for CHD systems \cite{yang2021_CHD_msav, yao2024_CHD_SAV, zheng2021_CHD_SAV}.
In the context of finite difference methods, both first- and second-order time-accurate energy-stable schemes have been constructed and analysed \cite{wise2010_CHD, chen2018_CHD}.
A rigorous analytical framework for the CHD system and its finite element approximation was established in \cite{feng_wise2012_CHD}, confirming the continuous and discrete energy dissipation properties.
Finally, a local DG method was proposed in \cite{guo2014_CHD_localDG}, which also preserves a discrete energy law while offering flexibility for complex geometries.


\subsection{Cahn--Hilliard--Navier--Stokes system}
Let $(\phi_{h}^0,\vv_{h}^0)\in \Vh\times\Xh$ be given. We seek function $(\phi_h,\vv_h)\in \Pi_c^1(\Itau;\Vh\times\Xh)$ and $(\mu_h,p_h)\in \Pi^0(\Itau;\Vh\times\Qh)$ such that
\begin{align}
  \la \dtau\phi_h,\psi_h \ra &- \la \phi_h^{n+1}\vv_h^{n+1}, \nabla\psi_h\ra + \la m(\phi_h^{n+1})\nabla\mu_h^{n+1},\nabla\psi_h \ra = 0, \label{eq:CHNSscheme1}
  \\
  \la \mu_h^{n+1},\xi_h \ra &- \gamma\la \nabla\phi_h^{n+1},\nabla\xi_h \ra - \la f'(\phi_h^{n+1},\phi_h^n),\xi_h \ra = 0,\label{eq:CHNSscheme2}
  \\
  \la \dtau\vv_h,\w_h \ra &+  \mathbf{c}_{skw}(\vv_h^{n+1},\vv_h^{n+1},\w_h) + \la \eta(\phi_h^{n+1})\mathrm{D}\vv_h^{n+1},\mathrm{D}\w_h \ra \notag
  \\
  &- \la p_h^{n+1},\div(\w_h)\ra + \la \phi_h^{n+1}\nabla\mu_h^{n+1},\w_h \ra = 0, \label{eq:CHNSscheme3}
  \\
  \la \div(\vv_h^{n+1}),q_h \ra & = 0 \label{eq:CHNSscheme4}
\end{align}
holds for all $(\psi_h,\xi_h,\w_h,q_h)\in\Vh\times\Vh\times\Xh\times\Qh$ and for all $0\leq n < n_T.$
Here we used the skew-symmetric form 
\begin{align*}
 \mathbf{c}_{skw}(\vv_h^*,\vv_h^{n+1},\w_h):= \tfrac{1}{2}\la (\vv_h^*\cdot\nabla)\vv_h^{n+1},\w_h \ra - \tfrac{1}{2}\la (\vv_h^*\cdot\nabla)\w_h,\vv_h^{n+1} \ra.      
\end{align*}

Similarly to CH and CHD, a growing body of work develops structure-preserving discretisations for the CHNS system. For the case of matching densities, energy-stable methods are proposed in e.g. \cite{feng2006fully, diegel2017convergence, kay2007efficient, kuang2025error}. For the case of non-matching densities modelled via quasi-incompressible models, we mention a nonlinear discontinuous Galerkin energy-stable discretisation \cite{giesselmann2015energy}, a linear conditionally energy-stable continuous Galerkin method \cite{shokrpour2018diffuse}, an unconditionally energy-stable nonlinear continuous Galerkin method \cite{brunk2025simple}, and a divergence-conforming discretisation \cite{ten2024divergence}. Furthermore, the works \cite{acosta2023property, guillen2024structure} propose methods that are both energy-stable and provide bounds on the phase-field parameter.  The related Abels--Garcke--Gr\"un model was studied in \cite{grun2013convergent, grun2014two, grun2016fully,chen2025optimal} with respect to an energy-stable and or positivity-preserving scheme. Finally, we note the work of Hong et al. \cite{hong2023physics} that provides a physics-informed, structure-preserving numerical scheme for ternary hydrodynamics, and the work of Khanwale et al. \cite{khanwale2022fully} for a discretisation of CHNS for turbulence.

\subsection{Mass balance and energy dissipation}
\begin{theorem}\label{thm:main} Solutions of all schemes conserve mass $ \la \phi_h^{n+1},1 \ra = \la \phi_h^n,1 \ra$.
\begin{enumerate}
    \item[] \textbf{CH}: Every discrete solution of scheme \eqref{eq:CHscheme1}--\eqref{eq:CHscheme2} satisfies
 \begin{align*}
   E_0(\phi_h^{n+1}) + \tau\la m(\phi_h^{n+1})\nabla\mu_h^{n+1},\nabla\mu_h^{n+1} \ra \leq E_0(\phi_h^{n}). 
 \end{align*}
 \item[] \textbf{CHD}: Every discrete solution of scheme \eqref{eq:CHDscheme1}--\eqref{eq:CHDscheme4} satisfies
 \begin{align*}
   E_0(\phi_h^{n+1}) + \tau\la m(\phi_h^{n+1})\nabla\mu_h^{n+1},\nabla\mu_h^{n+1} \ra + \tau\la \alpha(\phi_h^{n+1})\vv_h^{n+1},\vv_h^{n+1} \ra \leq E_0(\phi_h^{n}). 
 \end{align*}
 \item[] \textbf{CHNS}: Every discrete solution of scheme \eqref{eq:CHNSscheme1}--\eqref{eq:CHNSscheme4} satisfies
 \begin{align*}
   \!\!\!\!\!\!\!\!\!\!E_1(\phi_h^{n+1}\vv_h^{n+1}) \!+\! \tau\la m(\phi_h^{n+1})\nabla\mu_h^{n+1},\nabla\mu_h^{n+1} \ra \!+\! \tau\la \eta(\phi_h^{n+1})\mathrm{D}\vv_h^{n+1},\mathrm{D}\vv_h^{n+1} \ra \leq E_1(\phi_h^{n},\vv_h^{n}). 
 \end{align*}
\end{enumerate}
\end{theorem} 

\begin{proof} Mass conservation follows from $\psi_h=1$. 
For the energy we use $(a^2-b^2)= 2(a,a-b) - (a-b)^2$ compute
\begin{align*}
 \tfrac{1}{\tau}(E_\beta(\phi_h^{n+1},\vv_h^{n+1}) -E_\beta(\phi_h^{n},\vv_h^{n})) &= \gamma\la \nabla\phi_h^{n+1},\nabla \dtau\phi_h \ra + \beta\la \dtau\vv_h,\vv_h^{n+1} \ra 
 \\
 & - \frac{\tau\gamma}{2}\norm{\dtau\nabla\phi_h}_{L^2}^2 - \frac{\tau\beta}{2}\norm{\dtau\vv_h}_{L^2}^2
 \\
 &+ \frac{1}{\tau}\la f(\phi_h^{n+1})-f(\phi_h^n),1 \ra
\end{align*}

The second line can be estimated by zero and for the third line we use mean-value theorem to obtain
\begin{align*}
 \tfrac{1}{\tau}\la f(\phi_h^{n+1})-f(\phi_h^n),1 \ra &=  \tfrac{1}{\tau}\int_{t^n}^{t^{n+1}} \dt f(\phi^{n,n+1}_{h,\avg}(s)) \mathrm{d}s 
 \\
 &= \tfrac{1}{\tau}\int_{t^n}^{t^{n+1}}f'(\phi^{n,n+1}_{h,\avg}(s)) \dt \phi^{n,n+1}_{h,\avg}(s) \mathrm{d}s 
 \\
 &= \la \tfrac{1}{\tau}\int_{t^n}^{t^{n+1}}f'(\phi^{n,n+1}_{h,\avg}(s)) \mathrm{d}s,\dtau\phi_h \ra
 \\
 & = \la f'(\phi_h^{n+1},\phi_h^n),\dtau\phi_h \ra
\end{align*}
Using $\xi_h=\dtau\phi_h$, we obtain:
 \begin{align*}
   \tfrac{1}{\tau}(E_\beta(\phi_h^{n+1},\vv_h^{n+1}) -E_\beta(\phi_h^{n},\vv_h^{n})) \leq \la \dtau\phi_h, \mu_h^{n+1}\ra + \beta\la \dtau\vv_h,\vv_h^{n+1} \ra
 \end{align*} 
cf. \cite{brunk2024analysis,brunk2025analysis,brunk2025simple}, the energy inequalities follow from the same test function as at the continuous level.
\end{proof}

\section{Experiments}

The schemes of the previous section are implemented in NGSolve \cite{schberl2014c++11}. The nonlinear system is treated using Newton's method with an absolute tolerance of $10^{-11}$, while the linear system is solved by a direct solver. We consider the two-dimensional domain $\Omega = (0,1)^2$. In the experiments, we fix $\gamma = 10^{-4}, h=10^{-2}, \tau=10^{-3}$ and use the following functions that appear in the models:
\begin{align*}
	f(\phi) &= \tfrac{1}{4}\phi^2(1-\phi)^2, & m(\phi)&=5\max\{\phi^2(1-\phi)^2,0\}+ 10^{-6}, \\
	\alpha(\phi) &= \exp(\phi\ln(\alpha_1) + (1 - \phi)\ln(\alpha_0)), & \eta(\phi) &= \exp(\phi\ln(\eta_1) + (1 - \phi)\ln(\eta_0)),
\end{align*}
where $\alpha_1=1, \alpha_0=10^{-2} , \eta_1=10^{-2}, \eta_0=10^{-4}$. We consider the following initial data:
\begin{align*}
	\phi_0(x,y) &= 0.4 + \mathcal{U}(x,y), \qquad \mathcal{U}(x,y)\in[-10^{-3},10^{-3}], \qquad \vv_0(x,y) = (0,0)^\top.\label{eq:exp1}
\end{align*}
In Figure \ref{fig:evol}, we compare the temporal evolution of all three models. Until $t = 0.3$, all three models produce very similar solutions, corresponding to the stage in which the relevant interfaces have formed. At this point, CH relies solely on diffusive dynamics of the droplets, whereas the velocity-dependent models introduce additional transport effects, which in this case accelerate the phase separation.

\begin{figure}[htb!]
    \centering
    \includegraphics[page=1,width=.99\textwidth]{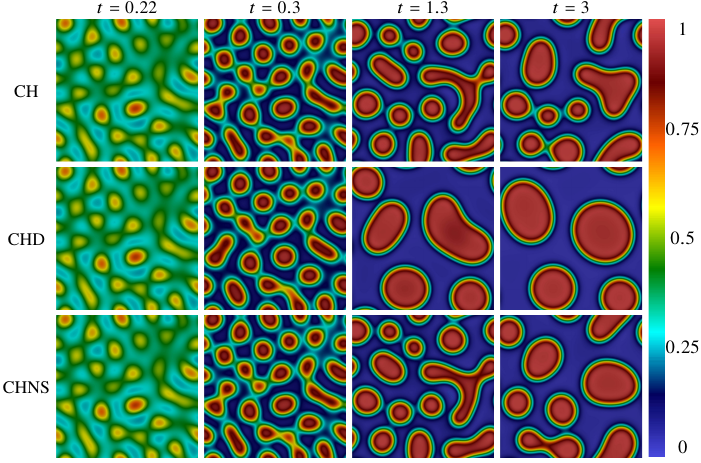}
    \caption{Snapshots of $\phi$ for CH, CHD, CHNS from top to bottom.
    }\vspace{-1em}
    \label{fig:evol}
\end{figure}

In Figure \ref{fig:mass_energy}, we compare the temporal evolution of the energy and the mass conservation error. As predicted by Theorem \ref{thm:main}, the conservation of mass holds up to machine accuracy, while the energy is dissipative. It is shown in Figure \ref{fig:mass_energy} that the energy of the CHNS model in the final stages is lower than that of the CH model, although there is a non-zero contribution of kinetic energy. Remarkably, the CHD system drives the phase separation much faster, such that even at $t = 1.3$ the process is significantly more advanced than in the other two models at $t=3$. It is important to recall that the velocity field is induced by the phase-field through the Korteweg stress.

\begin{figure}[htbp!]
\centering
\begin{subfigure}{0.49\textwidth}
\centering
\includegraphics[width=\textwidth]{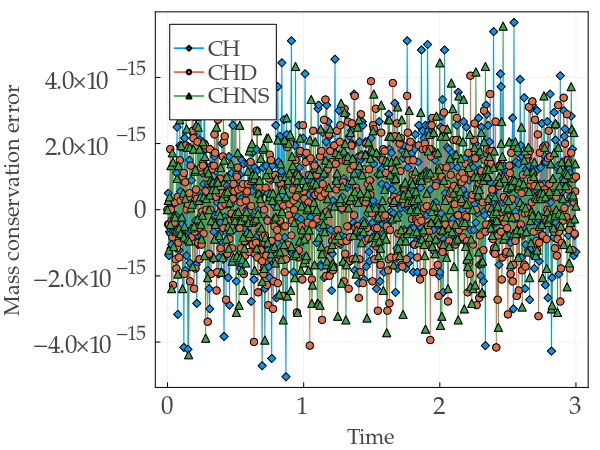}
\end{subfigure}
\begin{subfigure}{0.49\textwidth}
\centering
\includegraphics[width=\textwidth]{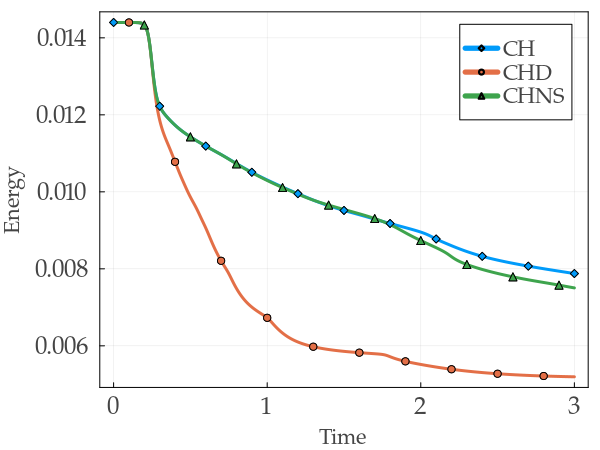}
\end{subfigure}
\caption{Comparison of mass conservation and energy evolution for CH, CHNS, CHD.}\vspace{-1em}
\label{fig:mass_energy}
\end{figure}

In Figure \ref{fig:special}, we compare the local divergence error, that is, $\lVert \nabla \cdot \vv \rVert_{L^2}$,  as well as the linear and angular momentum errors for CHNS. We observe that the divergence error remains in machine precision for CHD, while for CHNS it levels around $10^{-6}$. This behaviour is expected: CHD discretisation admits the choice $q_h = \nabla \cdot \mathbf{v}_h$, enforcing the freedom of local divergence, while CHNS discretisation does not, and thus it is not locally divergence free. Furthermore, linear and angular momentum are not preserved in CHNS, since the scheme is not designed to conserve them. \vspace{-1em}

\begin{figure}[htbp!]
\centering
\begin{subfigure}{0.49\textwidth}
\centering
\includegraphics[width=\textwidth]{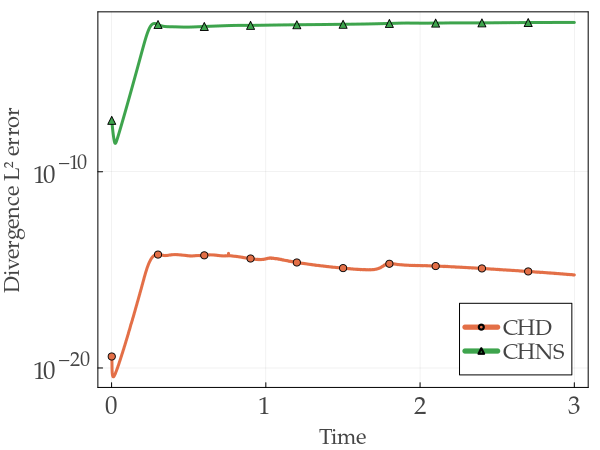}
\end{subfigure}
\begin{subfigure}{0.49\textwidth}
\centering
\includegraphics[width=\textwidth]{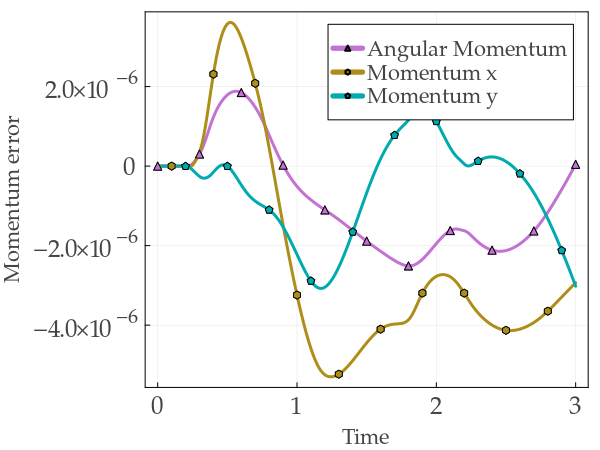}
\end{subfigure}
\caption{Comparison of local divergence-freedom for CHD and CHNS on logarithmic scale (left) and conservation errors of linear and angular momentum for CHNS (right).}\vspace{-1em}
\label{fig:special}
\end{figure}

\begin{acknowledgement}
\noindent The work of A.B. was supported by the Deutsche Forschungsgemeinschaft (DFG) via TRR 146, project number 233630050 and together with D.H. via SPP 2256 project number 441153493. MtE was supported by the DFG via project EI 1210/5-1, number 566600860. The author M.F.~is supported by the State of Upper Austria.
D.T.~was supported by the DFG through the Gra\-duier\-ten\-kolleg 2339 IntComSin (Project-ID 321821685), and by the Swedish Research Council (grant no.~2021-06594) during his stay at Institut Mittag-Leffler, Sweden, in 2025.
\end{acknowledgement}
\bibliographystyle{spmpsci}
\bibliography{literature.bib}

\end{document}